\documentclass[11pt]{amsart}
\usepackage{a4wide}
\usepackage[all]{xy}

\newcommand{\w}{\omega}
\newcommand{\defeq}{\overset{\mbox{\tiny\sf def}}=}
\newcommand{\F}{\mathcal F}

\newcommand{\M}{\mathcal M}
\newcommand{\U}{\mathcal U}

\newcommand{\IR}{\mathbb R}

\newcommand{\A}{\mathcal A}
\newcommand{\Tau}{\mathcal T}
\newcommand{\nw}{\mathsf{nw}}
\newcommand{\sw}{\mathsf{sw}}

\title{$Q$-spaces, perfect spaces and related cardinal characteristics of the continuum}
\author{Taras Banakh and Lidiya Bazylevych}
\address{Ivan Franko National University of Lviv, Ukraine}
\email{t.o.banakh@gmail.com}
\email{izar@litech.lviv.ua}

\newtheorem{theorem}{Theorem}
\newtheorem{corollary}[theorem]{Corollary}
\newtheorem{proposition}[theorem]{Proposition}
\newtheorem{problem}[theorem]{Problem}
\newtheorem{question}[theorem]{Question}

\newtheorem{lemma}[theorem]{Lemma}

\begin{document}
\begin{abstract} A topological space $X$ is called a {\em $Q$-space} if every subset of $X$ is of type $F_\sigma$ in $X$. For $i\in\{1,2,3\}$ let $\mathfrak q_i$ be the smallest cardinality of a second-countable $T_i$-space which is not a $Q$-space. It is clear that $\mathfrak q_1\le\mathfrak q_2\le\mathfrak q_3$. For $i\in\{1,2\}$ we prove that $\mathfrak q_i$ is equal to the smallest cardinality of a second-countable $T_i$-space which is not perfect. Also we prove that $\mathfrak q_3$ is equal to the smallest cardinality of a submetrizable space, which is not a $Q$-space. Martin's Axiom implies that $\mathfrak q_i=\mathfrak c$ for all $i\in\{1,2,3\}$.
\end{abstract}
\maketitle

A topological space $X$  is called 
\begin{itemize}
\item {\em perfect} if every open subset is of type $F_\sigma$ in $X$;
\item a {\em $Q$-space} if every subset is of type $F_\sigma$ in $X$;
\item a {\em non-$Q$-space} if $X$ is not a $Q$-space.
\end{itemize}
A subset of topological space is {\em of type $F_\sigma$} if it can be written as the union of countably many closed sets. 

It is clear that every $Q$-space is perfect. Under Martin's Axiom, every metrizable separable space of cardinality $<\mathfrak c$ is a $Q$-space, see \cite[4.2]{Miller}.

For a class $\Tau$ of topological spaces, denote by $\mathfrak q_{\Tau}$ the smallest cardinality of a non-$Q$-space $X\in\Tau$. The cardinal $\mathfrak q_{\Tau}$ is well-defined only for classes $\Tau$ containing  non-$Q$-spaces. In this paper we study the cardinals $\mathfrak q_\Tau$ for classes $\Tau$ of topological spaces satisfying various separation properties.

A topological space $X$ is called
\begin{itemize}
\item a {\em $T_1$-space} if every finite subset is closed in $X$;
\item a {\em $T_2$-space} if $X$ is {\em Hausdorff}, which means that any distinct points in $X$ have disjoint neighborhoods;
\item a {\em $T_{2\frac12}$-space} if $X$ is {\em Urysohn}, which means that any distinct points in $X$ have disjoint closed neighborhoods;
\item {\em functionally Hausdorff} if for any distinct points $x,y\in X$ there exists a continuous function $f:X\to\IR$ such that $f(x)\ne f(y)$;
\item a {\em $T_3$-space} if $X$ is Hausdorff and every neighborhood of any point $x\in X$ contains a closed neighborhood of $x$;
\item {\em Tychonoff} if $X$ is Hausdorff and for any closed set $F\subseteq X$ and point $x\in X\setminus F$ there exists a continuous function $f:X\to\IR$ such that $f(x)=1$ and $f[F]\subseteq\{0\}$;
\item {\em submetrizable} if there exists a continuous bijective map $f:X\to Y$ onto a metric space $Y$;
\item {\em second-countable} if $X$ has a countable base of the topology.
\end{itemize}
For every topological space we have implications
$$
\xymatrix{
\mbox{metrizable}\ar@{=>}[r]\ar@{=>}[d]&\mbox{submetrizable}\ar@{=>}[r]&\mbox{functionally}\atop\mbox{Hausdorff}\ar@{=>}[d]\\
\mbox{Tychonoff}\ar@{=>}[r]&\mbox{$T_3$-space}\ar@{=>}[r]&\mbox{Urysohn}\ar@{=>}[r]&\mbox{Hausdorff}\ar@{=>}[r]&\mbox{$T_1$-space}
}
$$
Observe that the one-point compactification of the discrete space of cardinality $\w_1$ is not perfect and hence $\mathfrak q_{\Tau}=\w_1$ for the class $\Tau$ of (compact) Tychonoff spaces of weight $\le\w_1$. For classes $\Tau$ of second-countable spaces the cardinals $\mathfrak q_\Tau$ are more interesting.
  
For $i\in\{1,2,2\frac12,3\}$ denote by $\mathfrak q_i$ the smallest cardinality of a second-countable $T_i$-space which is not a $Q$-space. It is clear that $\mathfrak q_1\le\mathfrak q_2\le\mathfrak q_{2\frac12}\le\mathfrak q_3$. Since each second-countable $T_3$-space (of cardinality $<\mathfrak c$) is metrizable (and zero-dimensional), the cardinal $\mathfrak q_3$ coincides with the well-known cardinal $\mathfrak q_0$, defined as the smallest cardinality of a subset of $\IR$ which is not a $Q$-space. The cardinal $\mathfrak q_0$ is well-studied in Set-Theoretic Topology, see \cite[\S4]{Miller}, \cite{Brendle} or \cite{Ban}.
This cardinal has the following helpful property.

\begin{proposition}\label{p:sm} Every submetrizable space of cardinality $<\mathfrak q_0$ is a $Q$-space.
\end{proposition}

\begin{proof} Let $X$ be a submetrizable space of cardinality $<\mathfrak q_0$. Find a  continuous bijective map $f:X\to Y$ onto a metrizable space $Y$. The metrizable space $Y$ has cardinality $|Y|=|X|<\mathfrak q_0\le\mathfrak c$ and weight $w(Y)\le\w\cdot|Y|\le\mathfrak c$. By \cite[4.4.9]{Eng}, $Y$ admits a topological embedding $h:Y\to J(\mathfrak c)^\w$ into the countable power of the hedgehog space $J(\mathfrak c)$ with $\mathfrak c$ many spikes. Here $J(\mathfrak c)$ is the set $\{x\in[0,1]^\mathfrak c:|\{\alpha\in\mathfrak c:x(\alpha)>0\}|\le 1\}$ endowed with the metric
$$d(x,y)=\max_{\alpha\in\mathfrak c}|x(\alpha)-y(\alpha)|.$$  
It is easy to see that the hedgehog space $J(\mathfrak c)$ admits a continuous bijective map onto the triangle $\{(x,y)\in[0,1]^2:x+y\le 1\}$ and hence $J(\mathfrak  c)^\w$ admits a continuous bijective map $\beta:J(\mathfrak c)^\w\to[0,1]^\w$ onto the Hilbert cube $[0,1]^\w$. Then $g\defeq\beta\circ f:X\to[0,1]^\w$ is a continuous injective map. The metrizable separable space $g[X]$ has cardinality $<\mathfrak q_0$ and hence is a $Q$-space. Then for every set $A\subseteq X$ it image $g[A]$ is of type $F_\sigma$ in $g[X]$. By the continuity of $g$, the preimage $g^{-1}[g[A]]=A$ of $g[A]$ is an $F_\sigma$-set in $X$, witnessing that $X$ is a $Q$-space.
\end{proof}

Now we prove some criteria of submetrizability among ``sufficiently small'' functionally Hausdorff spaces.

A family $\F$ of subsets of a topological space $X$ is called 
\begin{itemize}
\item {\em separating} if for any distinct points $x,y\in X$ there exists a set $F\in\F$ that contains $x$ but not $y$;
\item a {\em network} if  for every open set $U\subseteq X$ and point $x\in U$ there exists a set $F\in\F$ such that $x\in F\subseteq U$.
\end{itemize}

We say that a topological space $X$ is
\begin{itemize}
\item {\em Lindel\"of} if every open cover of $X$ has a countable subcover;
\item {\em hereditarily Lindel\"of} if every subspace of $X$ is Lindel\"of;
\item {\em $\nw$-countable} if $X$ has a countable network;
\item {\em $\sw$-countable} if $X$ has a countable separating family of open sets.
\end{itemize}

It is clear that every second-countable $T_1$-space is both $\nw$-countable and $\sw$-countable.

\begin{lemma}\label{l:hL} Every functionally Hausdorff space $X$ with hereditarily Lindel\"of square is submetrizable.
\end{lemma}

\begin{proof}
Since $X$ is functionally Hausdorff, for any distinct points $a,b\in X$, there exists a continuous function $f_{a,b}:X\to\IR$ such that $f_{a,b}(a)=0$ and $f_{a,b}(b)=1$.
By the continuity of $f_{a,b}$, the sets $$f_{a,b}^{-1}({\downarrow}\tfrac12)\defeq\{x\in X:f_{a,b}(x)<\tfrac12\}\quad\mbox{and}\quad f^{-1}_{a,b}({\uparrow}\tfrac12)\defeq\{y\in X:f_{a,b}(y)>\tfrac12\}$$are open neighborhoods of the points $a,b$, respectively. 
Since the space $X\times X$ is hereditarily Lindel\"of, the open cover
$\big\{f^{-1}_{a,b}({\downarrow}\tfrac12)\times f^{-1}_{a,b}({\uparrow}\tfrac12):(a,b)\in \nabla_X\}$ of the subspace $$\nabla_X\defeq\{(x,y)\in X\times X:x\ne y\}$$of $X\times X$ has a countable subcover. Consequently, there exists a countable set $C\subseteq \nabla_X$ such that $\nabla_X=\bigcup_{(a,b)\in C}f^{-1}_{a,b}({\downarrow}\tfrac12)\times f^{-1}_{a,b}({\uparrow}\tfrac12).$

Consider the metrizable  space $\IR^C$ and the continuous function
$$f:X\to\IR^C,\quad f:x\mapsto (f_{a,b}(x))_{(a,b)\in C}.$$
This function is injective because for any distinct $x,y\in X$ there exists a pair $(a,b)\in C$ such that $(x,y)\in f^{-1}_{a,b}({\downarrow}\tfrac12)\times f^{-1}_{a,b}({\uparrow}\tfrac12)$ and hence $f_{a,b}(x)<\frac12<f_{a,b}(y)$, witnessing that  $f(x)\ne f(y)$.
\end{proof}

\begin{corollary}\label{c:nw} Every $\nw$-countable functionally Hausdorff space of cardinality $<\mathfrak q_0$ is a submetrizable $Q$-space.
\end{corollary}

\begin{proof} Let $X$ be an $\nw$-countable functionally Hausdorff space. By \cite[3.8.12]{Eng}, the square $X\times X$ has countable network and is hereditarily Lindel\"of. By Lemma~\ref{l:hL}, $X$ is submetrizable and by Proposition~\ref{p:sm}, $X$ is a $Q$-space.
\end{proof}

Propositions~\ref{p:sm} and Corollary~\ref{c:nw} will help us to prove the following characterization of the cardinal $\mathfrak q_0$. 

\begin{proposition}\label{p:q0} The cardinal $\mathfrak q_0$ is equal to:
\begin{itemize}
\item the smallest cardinality of a submetrizable space which is not a $Q$-space;
\item the smallest cardinality of a non-perfect submetrizable space;
\item the smallest cardinality of an $\nw$-countable functionally Hausdorff non-$Q$-space;
\item the smallest cardinality of an non-perfect $\nw$-countable functionally Hausdorff space;
\item the smallest cardinality of second-countable functionally Hausdorff non-$Q$-space.
\item the smallest cardinality of non-perfect second-countable functionally Hausdorff space.
\end{itemize}
\end{proposition}

\begin{proof} Let 
\begin{itemize}
\item $\mathfrak q_{sm}$ be the smallest cardinality of a submetrizable non-$Q$-space;
\item $\mathfrak p_{sm}$ be the smallest cardinality of a non-perfect submetrizable space;
\item $\mathfrak q_{nw}$ be the smallest cardinality of an $\nw$-countable functionally Hausdorff non-$Q$-space;
\item $\mathfrak p_{nw}$ be the smallest cardinality of an non-perfect $\nw$-countable functionally Hausdorff space;
\item $\mathfrak q_{w}$ be the smallest cardinality of second-countable functionally Hausdorff non-$Q$-space.
\item$\mathfrak p_{w}$ be the smallest cardinality of non-perfect second-countable functionally Hausdorff space.
\end{itemize}
We should prove that all these cardinals are equal to $\mathfrak q_0$. The inclusions between corresponding classes of topological spaces yield the following diagram in which an arrow $\kappa\to\lambda$ between cardinals $\kappa,\lambda$ indicates that $\kappa\le\lambda$.
$$
\xymatrix{
\mathfrak p_{sm}\ar[r]&\mathfrak p_{nw}\ar[r]&\mathfrak p_{w}\\
\mathfrak q_{sm}\ar[r]\ar[u]&\mathfrak q_{nw}\ar[r]\ar[u]&\mathfrak q_{w}
\ar[u]}
$$
Proposition~\ref{p:sm} implies that $\mathfrak q_0\le\mathfrak q_{sm}$. To prove that all these cardinals are equal to $\mathfrak q_0$, it remains to prove that $\mathfrak p_w\le\mathfrak q_0$.

By the definition of the cardinal $\mathfrak q_0$, there exists a second-countable metrizable non-$Q$-space $X$ and hence $X$ contains  a subset $A$ which is not of type $G_\delta$ in $X$. Let $\tau'$ be the topology on $X$, generated by the subbase $\tau\cup\{X\setminus A\}$ where $\tau$ is the topology of the metrizable space $X$. It is clear that $X'=(X,\tau')$ is a second-countable space containing $A$ as a closed subset. Since $\tau\subseteq\tau'$, the space $X'$ is submetrizable and functionally Hausdorff. Assuming that $X'$ is perfect, we conclude that the closed set $A$ is equal to the intersection $\bigcap_{n\in\w}W_n$ of some open sets $W_n\in\tau'$. By the choice of the topology $\tau'$, for every $n\in\w$ there exists open sets $U_n,V_n\in \tau$ such that $W_n=U_n\cup(V_n\setminus A)$. It follows from $A\subseteq W_n=U_n\cup(V_n\setminus A)$ that $A=A\cap W_n=A\cap U_n\subseteq U_n$. Then 
$$A=\bigcap_{n\in\w}W_n=A\cap\bigcap_{n\in\w}W_n=\bigcap_{n\in\w}(A\cap W_n)=\bigcap_{n\in\w}(A\cap U_n)\subseteq \bigcap_{n\in\w}U_n\subseteq \bigcap_{n\in\w}W_n=A$$
and hence $A=\bigcap_{n\in\w}U_n$ is a $G_\delta$-set in $X$, which contradicts the choice of $A$. This contradiction shows that the functionally Hausdorff second-countable space $X'$ is not perfect and hence $\mathfrak p_w\le|X'|=\mathfrak q_0$.
\end{proof} 

Proposition~\ref{p:q0} suggests the following

\begin{question} Is $\mathfrak q_2=\mathfrak q_0$?
\end{question}

Repeating the argument of the proof of Proposition~\ref{p:q0}, we can prove the following characterization of the cardinals $\mathfrak q_i$ for $i\in\{1,2,2\frac12\}$.

\begin{proposition}\label{p:pi} Let $i\in\{1,2,2\frac12\}$. The cardinal $\mathfrak q_i$ is equal to
the smallest cardinality of a non-perfect second-countable $T_i$-space.
\end{proposition}
 
\begin{proposition}\label{p:sw} Every $\sw$-countable space of cardinality $<\mathfrak q_1$ is a $Q$-space.
\end{proposition}

\begin{proof} Let $X$ be an $\sw$-countable space of cardinality $<\mathfrak q_1$. By the $\sw$-countability of $X$, there exists a countable separating family $\U$ of open sets in $X$. Consider the topology $\tau$ on $X$ generated by the subbase $\U$ and observe that $X_\tau\defeq (X,\tau)$ is a second-countable $T_1$-space of cardinality $<\mathfrak q_1$. The definition of $\mathfrak q_1$ ensures that $X_\tau$ is a $Q$-space. Then every set $A\subseteq X$ is an $F_\sigma$ set in $X_\tau$. Since the identity map $X\to X_\tau$ is continuous, the set $A$ remains of type $F_\sigma$ in $X$, witnessing that $X$ is a $Q$-space.
\end{proof}

\begin{proposition} The cardinal $\mathfrak q_1$ is equal to
\begin{enumerate}
\item the smallest cardinality of an $\sw$-countable non-$Q$-space;
\item  the smallest cardinality of a non-perfect $\sw$-countable space.
\end{enumerate}
\end{proposition}

\begin{proof} Let 
\begin{itemize}
\item $\mathfrak q_{sw}$ be the smallest cardinality of an $\sw$-countable non-$Q$-space; 
\item $\mathfrak p_{sw}$ be  the smallest cardinality of a non-perfect $\sw$-countable space;
\item $\mathfrak p_{1}$ be the smallest cardinality of a non-perfect second-countable $T_1$-space;
\end{itemize} Propositions~\ref{p:sw}, \ref{p:pi} and the definitions of the cardinals $\mathfrak q_{sw},\mathfrak p_{sw},\mathfrak p_1$ imply that
$$\mathfrak q_1\le\mathfrak q_{sw}\le\mathfrak p_{sw}\le\mathfrak p_{1}=\mathfrak q_1$$ and hence $\mathfrak q_1=\mathfrak q_{sw}=\mathfrak p_{sw}$.
\end{proof}

Finally we establish a non-trivial lower bound on the cardinal $\mathfrak q_1$ using the cardinal characteristic $\mathfrak{adp}$, which is intermediate between the cardinals $\mathfrak{ap}$ and $\mathfrak{dp}$ introduced and studied by Brendle \cite{Brendle}. To introduce these cardinals, we need to recall three notions.

A family of sets $\A$ is called {\em almost disjoint} if $A\cap B$ is finite for any distinct sets $A,B\in\A$.

Two families of sets $\A,\mathcal B$ are called {\em orthogonal} if $A\cap B$ is finite for every $A\in\A$ and $B\in\mathcal B$.

We shall say that a family $\A$ of sets is {\em weakly separated} from a family of sets $\mathcal B$ if there exists a set $D$ such that for every $A\in\A$ the intersection $A\cap D$ is infinite and for every $B\in\mathfrak B$ the intersection $B\cap D$ is finite.

Observe that the notion of orthogonality is symmetric whereas  the weak separatedness is not.

For any set $X$ we denote by $[X]^\w$ and $[X]^{<\w}$ the families of infinite and finite subsets of $X$, respectively.

Let 
\begin{itemize}
\item $\mathfrak{ap}$ be the smallest cardinality of an almost disjoint family $\A\subseteq[\w]^\w$ that contains a subfamily $\mathcal B\subseteq\A$ which cannot be weakly separated from $\A\setminus\mathcal B$;
\item $\mathfrak{dp}$ be the smallest cardinality of the union $\A\cup\mathcal B$ of two orthogonal families $\A,\mathcal B\subseteq[\w]^\w$ such that $\A$ cannot be weakly separated from $\mathfrak B$;
\item $\mathfrak{adp}$ be the smallest cardinality of the union $\A\cup\mathcal B$ of two orthogonal families $\A,\mathcal B\subseteq[\w]^\w$ such that $\A$ is almost disjoint and cannot be weakly separated from $\mathcal B$. 
\end{itemize}

It is clear that $$\mathfrak{dp}\le\mathfrak{adp}\le\mathfrak{ap}.$$ According to \cite{Brendle} we also have the inequalities
$$\mathfrak p\le\mathfrak{dp}\le\mathfrak{ap}\le\min\{\mathfrak q_0,\mathrm{add}(\mathcal M)\},$$
where $\mathrm{add}(\M)$ is the smallest cardinality of a family $\A$ of meager subsets of the real line whose union $\bigcup\A$ is not meager in $\mathbb R$, and $\mathfrak p$ is the smallest cardinality of a family $\mathcal B\subseteq[\w]^\w$ such that for every finite subfamily $\F\subseteq\mathcal B$ the intersection $\bigcap\F$ is infinite but for every infinite set $I\subseteq\w$ there exists a set $F\in\F$ such that $I\cap F$ is finite. It is well-known (see \cite{Blass} or \cite{Vaughan})  that $\mathfrak p=\mathfrak c$ under Martin's Axiom, where $\mathfrak c$ stands for the cardinality of continuum.

\begin{proposition}\label{p:adp-q} Every second-countable $T_1$-space $X$ of cardinality $|X|<\mathfrak{adp}$ is a $Q$-space and hence $$\mathfrak{p}\le\mathfrak{dp}\le\mathfrak{adp}\le\mathfrak q_1\le\mathfrak q_2\le\mathfrak q_{2\frac12}\le\mathfrak q_3=\mathfrak q_0.$$
\end{proposition}

\begin{proof} Given any subset $A\subseteq X$, we should prove that $A$ is of type $G_\delta$ in $X$. Let $\mathcal B=\{U_n\}_{n\in\w}$ be a countable base of the topology of the space $X$.

For every $y\in X\setminus A$, let $I_y=\{n\in \w:y\in U_n\}$. Since $\{U_n\}_{n\in\w}$ is a base of the topology of $X$, for every $x\in A$ there exists an infinite set $I_x\subseteq\w$ satisfying two conditions:
\begin{itemize}
\item for any numbers $n<m$ in $I_x$ we have $x\in U_m\subseteq U_n$;
\item for every neighborhood $O_x$ of $x$ in $X$ there exists $n\in I_x$ such that $x\in U_n\subseteq O_x$.
\end{itemize}
We claim that for any $x\in A$ and $y\in B$ the intersection $I_x\cap I_y$ is finite.
Indeed, by the choice of the  set $I_x$, there exists $n\in\w$ such that $U_n\subseteq X\setminus\{y\}$. Then for every $m\ge n$ we have $U_m\subseteq U_n$
and hence $y\notin U_m$ and $m\notin I_y$. Therefore, the families $\{I_x:x\in A\}$ and $\{I_y:y\in X\setminus A\}$ are orthogonal. The same argument shows that the family $\{I_x:x\in A\}$ is almost disjoint.

Since $|A\cup B|=|X|<\mathfrak{adp}$, the family $\{I_x:x\in A\}$ can be weakly separated from the family $\{I_y:y\in X\setminus A\}$ and hence  there exists a set $D\subseteq\w$ such that for any $x\in A$ the intersection $I_x\cap D$ is infinite and for any $x\in B$ the intersection $I_y\cap D$ is finite. For every finite set $F\subseteq D$ consider the open subset $$W_F\defeq\bigcup_{n\in D\setminus F}U_n$$ of $X$. For every $x\in A$ the infinite set $I_x\cap D$ contains a number $n\notin F$ and then $x\in U_n\subseteq W_F$. Therefore $G\defeq\bigcap_{F\in[D]^{<\w}}W_F$ is a $G_\delta$-set containing $A$. On the other hand, for every $x\in X\setminus A$, the intersection $F=I_y\cap D=\{n\in D:y\in U_n\}$ is finite and hence $y\notin \bigcup_{n\in D\setminus F}U_n=W_F$. Therefore, $A=\bigcap_{F\in[D]^{<\w}}W_F$ is a $G_\delta$ set in $X$ witnessing that $X$ is a $Q$-space.
\end{proof}

Since $\mathfrak p=\mathfrak c$ under Martin's Axiom, Proposition~\ref{p:adp-q} implies the following corollary.

\begin{corollary} Under Martin's Axiom, $\mathfrak q_i=\mathfrak c$ for every $i\in\{0,1,2,2\frac12,3\}$.
\end{corollary}

It would be interesting to have any additional information on (im)possible inequalities between the cardinals $\mathfrak q_i$ and other cardinal characteristics of the continuum. In particular, the following questions are natural and seem to be  open.

\begin{problem}
\begin{enumerate}
\item Is $\mathfrak{ap}\le\mathfrak q_2$? 
\item Is $\mathfrak q_1\le\mathrm{add}(\mathcal M)$?
\item Is $\mathfrak q_1=\mathfrak q_2$? 
\item Is the strict inequality $\mathfrak q_1<\mathfrak q_0$ consistent?
\end{enumerate}
\end{problem}

Also the position of the new cardinal $\mathfrak{adp}$ in the interval $[\mathfrak{dp},\mathfrak{ap}]$ is not clear.

\begin{problem}\label{prob:adp}
\begin{enumerate}
\item Is $\mathfrak{adp}=\mathfrak{dp}$ in ZFC?
\item Is $\mathfrak{adp}=\mathfrak{ap}$ in ZFC?
\end{enumerate}
\end{problem}
By \cite{Brendle}, the strict inequality $\mathfrak{dp}<\mathfrak{ap}$ is consistent, so one of the questions in Problem~\ref{prob:adp} has negative answer. But which one? Or both?

\end{document}